\documentclass{ws-ccm}
\usepackage{mathrsfs}
 \usepackage[colorlinks=true, pdfstartview=FitV, linkcolor=blue, citecolor=blue, urlcolor=blue]{hyperref}
\usepackage{amsfonts}
\usepackage{amssymb,xspace,enumerate,amscd,graphicx,color}
\usepackage[active]{srcltx}
\numberwithin{equation}{section}
\newcommand{\thmref}[1]{Theorem~\ref{#1}}

\newcommand{\Corref}[1]{Corollary~\ref{#1}}
\newcommand{\secref}[1]{\S\ref{#1}}
\newcommand{\lemref}[1]{Lemma~\ref{#1}}
\newcommand{\eqnref}[1]{~(\ref{#1})}

\def\s{\sigma}

\def\s{\sigma}
\def\Der{\operatorname{Der}}
\def\Aut{\operatorname{Aut}}
\def\Span{\operatorname{span}}
\def\Vir{\operatorname{Vir}}
\newcommand{\C}{\ensuremath{\mathbb C}\xspace}

\renewcommand{\a}{\ensuremath{\alpha}}
\renewcommand{\b}{\ensuremath{\beta}}

\newcommand{\N}{\mathbb{N}}
\newcommand{\Z}{\ensuremath{\mathbb{Z}}\xspace}

\newcommand{\W}{\ensuremath{\mathscr{V}}\xspace}
\newcommand{\R}{\ensuremath{\mathbb{R}}\xspace}

\renewcommand{\phi}{\varphi}

\renewcommand{\leq}{\leqslant}
\renewcommand{\geq}{\geqslant}

\newcommand{\p}{\partial}

\begin{document}

\markboth{BEN COX, XIANGQIAN GUO, RENCAI LU, KAIMING ZHAO}
{$N$-point Virasoro algebras}

%
\catchline{}{}{}{}{}
%

\title{$N$-POINT VIRASORO ALGEBRAS AND THEIR MODULES OF DENSITIES}

\author{BEN COX}

\address{Department of Mathematics,  College of
Charleston,  \\
Charleston, SC 29424, USA.  \\
 coxbl@cofc.edu}

\author{XIANGQIAN GUO}

\address{Department of Mathematics,  Zhengzhou
University, \\
 Zhengzhou 450001, Henan, P. R. China. \\
guoxq@zzu.edu.cn
}
\author{RENCAI LU}

\address{Department of Mathematics, Suzhou University,   \\
Suzhou 215006, Jiangsu, P. R. China.  \\
rencail@amss.ac.cn}

\author{KAIMING ZHAO}

\address{Department of Mathematics, Wilfrid
Laurier University, \\  Waterloo, ON, Canada N2L 3C5,  \\ 
and College of
Mathematics and Information Science, Hebei Normal (Teachers) \\
University, Shijiazhuang, Hebei, 050016 P. R. China. \\
kzhao@wlu.ca
}

\maketitle

\begin{history}
\received{(Day Month Year)}
\revised{(Day Month Year)}
\end{history}

\begin{abstract}
 In this paper we introduce and study $n$-point Virasoro algebras, $\tilde{\W_a}$,
 which are natural generalizations
 of the classical Virasoro algebra and have as quotients multipoint genus zero Krichever-Novikov type algebras. We determine necessary and sufficient conditions for the latter two such Lie algebras
 to be isomorphic.  Moreover we determine their automorphisms, their derivation algebras, their universal
 central extensions,  and some other
 properties.   The  list of automorphism groups that occur is $C_n$, $D_n$, $A_4$, $S_4$ and $A_5$. We also construct a large class of modules which we call modules of densities, and
 determine necessary and sufficient conditions for them to be irreducible.
\end{abstract}



\section{Introduction}	
We assume in this paper that  $\mathbb N=\{1,2,\dots\}$ and $\mathbb Z_+=\{0,1,2,3,\dots\}$, and all vector spaces and algebras are over the complex numbers $\C$.

Consider the Laurent polynomial ring $\C[t, t^{-1}]$ as the ring
of rational functions on the Riemann sphere $\C \cup \{\infty\}$
with poles allowed only in $\{\infty, 0\}$. This geometric point of
view suggests a natural generalization of the loop algebra
construction. Instead of the sphere with two punctures, one can
consider any complex algebraic curve $X$ of genus $g$ with a fixed
subset $P$ of $n$ distinct points. Following this idea one arrives at M. Schlichenmaier's definition of multipoint algebras of Krichever-Novikov affine type if we replace $\C[t, t^{-1}]$ with the ring $R$ of
meromorphic functions on $X$ with poles allowed only in $P$ in the
construction of affine Kac-Moody algebras (see \cite{MR2058804},  \cite{MR902293}, \cite{MR925072}, and \cite{MR998426}). The $n$-point affine Lie algebras also appeared in the
work of Kazhdan and Lusztig (\cite[Sections 4 \& 7]{MR1104840},\cite[Chapter 12]{MR1849359}).  Krichever-Novikov algebras are used to constuct analogues of important mathematical objects used in string theory but in the setting of a Riemann surface of arbitrary genus: vacuum vectors, the stress-energy tensor,  normal ordering,  and operator product expansion formulae. Moreover Wess-Zumino-Witten-Novikov theory and analogues of the Knizhnik-Zamolodchikov equations are developed for analogues of the affine and Virasoro algebras (see the survey article \cite{MR2152962}, and for example  \cite{MR1706819}, \cite{MR2106647}, \cite{MR2072650}, \cite{MR2058804}, \cite{MR1989644}, and \cite{MR1666274}).
For a $q$-analogue of the Virasoro algebra see \cite{MR1280096}.

In this paper we study the algebra of derivations
of $R$ where $R$ is the ring of rational functions on $\mathbb P^1$ with poles allowed only at the fixed finite number of points
in $P$. Below we explicitly describe the universal central extension of this algebra which is a good algebraic candidate for the $n$-point Virasoro algebra (see also \cite{MR2035385}).

Let $n\in\N$.  Let $a_1, a_2, \cdots,
a_n\in\C$ be $n$ distinct  numbers, and let $\C(t)$ be the field of rational functions in
the indeterminate $t$. We set $a=(a_1, a_2, \cdots, a_n)\in\C^n$.
Let $R_a$ be the subalgebra of $\C(t)$ generated by $t,
(t-a_1)^{-1}, \cdots, (t-a_n)^{-1}$, and $\mathscr{V}_a=\Der(R_a)$. The universal central extension of the Lie algebra $\mathscr{V}_a$ we will call the {\it $(n+1)$-point
Virasoro algebra} of $\mathscr{V}_a=\Der(R_a)$ and we will denote it by $\tilde{\mathscr{V}_a}$. Perhaps one could denote $\Der(R_a)$ by $\mathscr{W}_a$, but this would likely cause a conflict of notation with $W$-algebras studied elsewhere.  When $n=1$ and $a_1=0$, $R_a$ is the ring of rational functions with poles at the two points $0$ and $\infty$ which is why we call the classical Virasoro algebra, a $2$-point Virasoro algebra. 

The present paper is organized as follows. In Section 2, we obtain
some basic properties of the associative algebra $R_a$ and the Lie
algebra $\W_a$; determine the group of all invertible elements of
$R_a$, give a useful basis of $R_a$ and $\W_a$ and explicitly
describe Lie brackets of $\W_a$  using its basis elements. The main
technique we use here is residue calculus from complex analysis. The Lie
brackets of $\W_a$ can give a lot of nontrivial combinatorial
identities, and we obtain one for later use. In Section 3, we obtain
the necessary and sufficient conditions for two Lie algebras $\W_a$
and $\W_{a'}$ to be isomorphic. The automorphism group $\text{Aut}\W_a$ is isomorphic to the automorphism group of the ring $R_a$ and it is the latter that we first explicitly describe.  In this regards one  should recall that in the classical affine setting $\text{Aut}\mathbb C[\![t]\!]$ gives by providing transformation rules, global geometric meaning to vertex operators on algebraic curves (see \cite[Chapter 5]{MR1849359}).  There are actually two distinct classes of isomorphisms in $\text{Aut}\,R_a$. We also determine the
isomorphism classes of all algebras $\W_a$ and describe the
automorphism group of any $\W_a$. Various examples are presented which
illustrate these results.   Using a result due to F. Klein \cite{MR0080930}, we describe the possible list of automorphism groups as that consisting of the cyclic group $C_N$, the dihedral group $D_N$, the alternating group $A_4$ or $A_5$, and the symmetric group $S_4$.  We give examples of $a$ showing that each of these groups occurs.  Rather mysteriously (at least to us) these are the same groups that appear elsewhere in Lie theory (e.g. in the McKay correspondence \cite{MR604577} and in conformal field theory \cite{MR918402}).  The second cohomology group of $\W_a$ is
computed in Section 4, and is shown to be $n$-dimensional.
Consequently, we can obtain the $n$-dimensional universal central
extension of $\W_a$ which we will call the $n$-point Virasoro
algebra. (One can deduce some combinatorial identities when
calculating the $2$-cocycles, which may be interesting to
combinatorists.) In the last section, we construct a large
class of modules $V(\a,\beta):=R_az$ over $\W_a$ parameterized by
$\a\in\Z^n$ and $\beta\in\C$. We determine necessary and sufficient
conditions for these modules to be irreducible. It is shown that the modules
 are irreducible except when (1) $\beta=0$ and $\a\in\Z^n$; or (2)
$\beta=1, n\geq 2$; or (3) $\beta=1$, $n=1$ and $\a\in\Z$. More precisely,
if $\beta=0$ and $\a\in\Z^n$, then $V(\a,\beta)$ has a nontrivial
irreducible quotient module and the corresponding submodule is
$1$-dimensional. Let $\partial=d/dt$.  If $\beta=1$, then $V(\a,\beta)$ has a smallest nonzero
submodule
$$
\p(R_az)=\{\p(gz):=\p(g)z+g\sum_{i=1}^n\frac{\a_i}{t-a_i}z\,|\,g\in R_a\},
$$
which is irreducible; $V(\a,\beta)/\p(R_az)$ is a trivial module and, is
$0$ if and only if $n=1$ and $\a\not\in \mathbb Z$. In addition if $\a\in
\Z^n$, then
$$\p(R_az)=\sum_{k\in\Z_+}\C t^kz\oplus\sum_{i=1}^n\sum_{k\in\N}\C¡¡(t-a_i)^{-k-1}z,$$
and $\dim(R_az/\p(R_az))=n$.

Since $n$-point Virasoro algebras have many properties
similar to the classical Virasoro algebra, we hope that they will have
applications to physics and to other areas of mathematics, just as in the case of the classical Virasoro algebra.

In some related work of the first author and V. Futorny \cite{CF1} a description is given for the generators and relations of the universal central extension of the infinite dimensional Lie algebra, $\mathfrak{sl}_2(\mathbb C)\otimes \mathbb C[t,t^{-1},u|u^2=(t^2-b^2)(t^2-c^2)]$, appearing in the work of Date, Jimbo, Kashiwara and Miwa in their study of integrable systems arising from the Landau-Lifshitz differential equation.    Here the universal central extension is described in terms of elliptic integrals and polynomials related to certain ultraspherical polynomials.
We are currently investigating the structure of the universal central extension and the automorphism group of the Lie algebra of differeomorphisms of such coordinate rings as  $\mathbb C[t,t^{-1},u|u^2=(t^2-b^2)(t^2-c^2)]$.  Again elliptic integrals and orthogonal polynomials make their appearance in this description.  This work will appear in a later publication.

Finally, the reader should note that since $\Der(R_a)=R_a\partial$, a result of Jordan \cite{MR829385} shows that
 $\Der(R_a)$ provides an interesting family of infinite-dimensional simple Lie algebras.

\section{Properties of $\mathscr{V}_a$}
We first recall some notation from \cite{MR966871}. Let $A$ be any unital
commutative associative algebra and let $\Der(A)$ is the Lie algebra of
all derivations of $A$. A Lie subalgebra $L$ of $\Der(A)$ is called
{\bf regular} if it is also an $A$-module.

Now we return to $\W_a$ for some $a=(a_1,\cdots,a_n)\in\C^n$ with
$a_i$ distinct. It is easy to see that $\mathscr{V}_a=R_a\partial$
which is a regular Lie algebra. Let $R_a^*$ be the unit group of
$R_a$, that is, $R_a^*$ consists of all invertible elements in
$R_a$. Similarly we define $\mathbb C^*$. 
We collect some properties of $\W_a$ as follows, which will be
repeatedly used later.

\begin{lemma}\label{lemma1} Let $R_a$ and $\W_a$ be as above. Then
\begin{enumerate}
\item[(a).] $R_a$ has a basis
$\{t^k,(t-a_1)^{-l}, \cdots, (t-a_n)^{-l} \,|\,k\in\Z_+, l\in\N\};$
\item[(b).] $\mathscr{V}_a$ has
a basis $\{t^k\partial,(t-a_1)^{-l}\partial, \cdots,
(t-a_n)^{-l}\partial \,|\,k\in\Z_+, l\in\N\};$
\item[(c).] $R^*_a=\{c\prod_{i=1}^n(t-a_i)^{k_i}\,|\, c\in\C^*, k_i\in\Z\}$. \end{enumerate}
\end{lemma}

\begin{proof}(a) follows from the Chinese Remainder Theorem.  (b) follows from (a) and the fact that
$\W_a=R_a\partial$ and $(R_a\partial)(t)=R_a$.

Now we prove (c). For any $x\in R_a^*$, we can write $x$ as
$x=\frac{f(t)}{g(t)},$ where $f(t)$ and $g(t)$ are relatively prime
polynomials in $t$, and that $g(t)=(t-a_1)^{k_{1}}\cdots
(t-a_n)^{k_{n}},$ for some $k_{i}\in\Z_+$.

Since $1/x=g/f\in R_a^*$, similar arguments infer that $f$ is also
of the form $f=c(t-a_1)^{k'_{1}}\cdots(t-a_n)^{k'_{n}}$ for some
$k'_{i}\in\Z_+$ and $c\in\C^*$.  This completes the proof.
\end{proof}

Now we determine the brackets of basis elements of  $\mathscr{V}_a$.
We see that  $\mathscr{V}_a$ has subalgebras isomorphic to the
centerless Virasoro algebra:
$$\Vir^{(i)}=\Span\{(t-a_i)^{k+1}\partial\,|\,k\in\Z\},$$
for any $i=1,2,\cdots ,n$. Note that the positive part
$\Span\{(t-a_i)^{k}\partial\,|\,k\in\Z_+\}$ of $\Vir^{(i)}$ is a
subalgebra of $\mathscr{V}_a$ which is independent of $i$. The Lie
bracket of $\W_a$ is defined as follows
\begin{equation}\label{definingbracket}
[f(t)\p, g(t)\p]=(f(t)\p(g(t))-g(t)\p(f(t)))\p,\,\,\forall\, f(t),g(t)\in R_a.
\end{equation}
For the convenience of later use, we write the brackets in terms of
basis elements of $\W_a$:

\begin{theorem}\label{algiso} For any $k,l\in\N$, $m\in\Z$ and $i\ne j$, we have

\begin{enumerate}
\item[(a).] \hspace{0.1cm}$[(t-a_i)^{k}\partial, (t-a_j)^{m}\partial]$
$$=\sum_{s=0}^{k}{{k}\choose{s}}(m+s-k)(a_j-a_i)^s(t-a_j)^{k+m-s-1}\p;\,\,\,\,\,\,\,\,\,\,\,\,\,\,\,\,\,$$

\item[(b).]
\begin{align*}
&[(t-a_i)^{-k}\partial, (t-a_j)^{-l}\partial] \\
&=\sum_{m=1}^{k+1}(2k+1-m){k+l-m\choose k+1-m}\frac{(t-a_i)^{-m}}{(a_i-a_j)^l(a_j-a_i)^{k+1-m}}\p \\
&\quad-\sum_{m=1}^{l+1}(2l+1-m){l+k-m\choose l+1-m}\frac{(t-a_j)^{-m}}{(a_j-a_i)^k(a_i-a_j)^{l+1-m}}\p.
\end{align*}
\end{enumerate}

\end{theorem}
One might want to compare the above with \cite[Section 8]{MR1039524}.
\begin{proof}
(a) is a direct consequence of Newton's binomial theorem:
\begin{align*}
[(t-a_i)^{k}&\partial,
(t-a_j)^{m}\partial] \\
&=(m(t-a_j)^{m-1}(t-a_i)^k-k(t-a_i)^{k-1}(t-a_j)^{m})\p \\
&=m\sum_{s=0}^k{{k}\choose{s}}(a_j-a_i)^s(t-a_j)^{k-s+m-1}\p \\
&
\hskip 2cm
-k\sum_{s=0}^{k-1}{{k-1}\choose{s}}(a_j-a_i)^s(t-a_j)^{k-s+m-1})\p \\
&=\sum_{s=0}^{k}{{k}\choose{s}}(m+s-k)(a_j-a_i)^s(t-a_j)^{k+m-s-1}\p.
\end{align*}

Now we prove (b). For any $b\neq c\in\C$, we know that
$$(t-c)^{-k}(t-b)^{-l}=\sum_{i=1}^kc_i(t-c)^{-i}+\sum_{j=1}^lb_j(t-b)^{-j},$$
for some complex numbers $c_i, b_j$.

Let $\gamma$ be a circle in the complex plane about $c$ with $b$
lying outside $\gamma$. For any $1\leq m\leq k$, consider the
following integral
$$\oint_\gamma(t-c)^{m-k-1}(t-b)^{-l}dt=\oint_\gamma(t-c)^{m-1}(\sum_{i=1}^kc_i(t-c)^{-i}+\sum_{j=1}^lb_j(t-b)^{-j})dt,$$
which gives that
$$
{{-l}\choose{k-m}}(c-b)^{m-l-k} =\frac{1}{(k-m)!}\p^{(k-m)}((t-b)^{-l})|_{t=c}= c_{m}.$$
A similar calculation gives
$$
b_{m}={{-k}\choose{l-m}}(b-c)^{m-k-l},\,\,\forall\, 1\leq m\leq l.$$
Thus
\begin{align*}
(t-c)^{-k}(t-b)^{-l} =\sum_{m=1}^k{{-l}\choose{k-m}}\frac{(c-b)^{m-l-k}}{(t-c)^{m}}+\sum_{m=1}^l{{-k}\choose{l-m}}\frac{(b-c)^{m-k-l}}{(t-b)^{m}}.
\end{align*}
Substituting it into the following Lie bracket, we  deduce
\begin{align*}
[(t-a_i)^{-k}&\partial,(t-a_j)^{-l}\partial] \\
&=k(t-a_i)^{-k-1}(t-a_j)^{-l}\p-l(t-a_i)^{-k}(t-a_j)^{-l-1}\p \\
&=\sum_{m=1}^{k+1}(2k+1-m){k+l-m\choose k+1-m}\frac{(t-a_i)^{-m}}{(a_i-a_j)^l(a_j-a_i)^{k+1-m}} \\
&\enspace-\sum_{m=1}^{l+1}(2l+1-m){l+k-m\choose l+1-m}\frac{(t-a_j)^{-m}}{(a_j-a_i)^k(a_i-a_j)^{l+1-m}}.
\end{align*}
Note that the above formula is also true for $k=0$ or $l=0$ if we
make the convention that ${{m-1}\choose{m}}=\delta_{m,0}$. This
completes the proof of (b).
%
%
\end{proof}

From the above Lie bracket, we can deduce the following combinatorial
identity, which could be used in Section 4 to verify that the $\phi_i$ in that section are $2$-cocycles, but we will follow another route to the proof of this fact.

\begin{corollary}\label{combinatorialid} Suppose that $x,y,z\in\C$ are distinct and
$m,k,l\in\N$, then for any $1\leq r\leq m+1$ we have the following
identity
$$\sum_{s=1}^{k+1}{k+l-s\choose k+1-s}{m+s-r\choose m+1-r}\frac{(-1)^{l+s}(2k+1-s)(2m+1-r)}{(z-y)^{l+k+1-s}(y-x)^{m+s+1-r}}$$
$$\hspace{0.5cm}-\sum_{s=1}^{l+1}{l+k-s\choose l+1-s}{m+s-r\choose m+1-r}\frac{(-1)^{k+s}(2l+1-s)(2m+1-r)}{(y-z)^{l+k+1-s}(z-x)^{m+s+1-r}}$$
$$=\sum_{s=r-1}^{m+1}{m+k-s\choose m+1-s}{s+l-r\choose s+1-r}\frac{(-1)^{k+l}(2m+1-s)(2s+1-r)}{(y-x)^{k+m+1-s}(z-x)^{l+s+1-r}}$$
$$\hspace{0.5cm}-\sum_{s=r-1}^{m+1}{m+l-s\choose m+1-s}{s+k-r\choose s+1-r} \frac{(-1)^{k+l}(2m+1-s)(2s+1-r)}{(z-x)^{l+m+1-s}(y-x)^{k+s+1-r}}.$$
\end{corollary}

\begin{proof} Denote $a=(x,y,z)\in\C$ and consider the Lie algebra
$\W_a$, then $(t-x)^{-m}, (t-y)^{-k}, (t-z)^{-l}$ are linearly
independent over $\C$. Fix some $m,k,l\in\N$. Now consider the
Jacobian identity
$$[(t-x)^{-m}\p, [(t-y)^{-k}\p,
(t-z)^{-l}\p]]$$
$$=[[(t-x)^{-m}\p, (t-y)^{-k}\p], (t-z)^{-l}\p]+[(t-y)^{-k}\p, [(t-x)^{-m}\p, (t-z)^{-l}\p]].$$
Expanding all the brackets above using \thmref{algiso} (b),
and comparing the coefficients of $(t-x)^{-r}$ for any $1\leq r\leq
m+1$, we get the identity in the lemma. We omit the details.
\end{proof}

%

\section{Isomorphisms, automorphisms and Derivations}

We first determine the conditions for two different $n$-point
Virasoro algebras to be isomorphic.
\begin{theorem}\label{ringisothm}
Suppose $\{a_1, a_2, \cdots, a_n\}$ and $\{a'_1, a'_2, \cdots,
a'_m\}$ are two sets of distinct complex numbers and $\phi: R_a\to
R_{a'}$ is an isomorphism of associative algebras. Then $m=n$, $\phi$ is a fractional linear transformation, and
one of the following two cases holds
\begin{enumerate}
\item[(a).] There exists some constant $c\in\C$ such that
$a_i-a_1=c(a'_i-a'_1)$ and $\phi((t-a_i)^{k} )=c^{k}(t-a'_i)^{k} $
for all $k\in\Z$ and $i=1,2,\cdots ,n$ after reordering $a'_i$ if
necessary.
\item[(b).] There exists some constant $c\in\C$ such that
$(a_i-a_1)(a'_i-a'_1)=c$ for all $i\ne1$, and
$$\phi((t-a_1)^k)=\frac{c^{k}}{(t-a'_1)^{k}},$$
$$\phi((t-a_i)^k)=\frac{(a_1-a_i)^k(t-a'_i)^k}{(t-a'_1)^{k}}, \,\,\forall\,\,i>1,$$
for all $k\in\Z$ after reordering $a_i$ and $a'_i$ if necessary.
\end{enumerate}
Moreover, the maps given in (a) and (b) indeed define algebra
isomorphisms between $R_a$ and $R_{a'}$ under the corresponding
conditions.
\end{theorem}
The fact that any automorphism of $\mathbb P^1$ is completely determined by where it sends $0,1$ and $\infty$,
is due to the fact that the Galois group of $\mathbb C(t)$ consists of all M\"obius inversions
(see \cite[Example 21.9]{MR1243417}).

\begin{proof}  We know that  $\phi(R_a^*)= R_{a'}^*$ as multiplicative
groups, and $\phi(\C^*)$ $=\C^*$. By \lemref{lemma1}, we have that
$$R^*_a=\{c\prod_{i=1}^n(t-a_i)^{k_i}\,|\, c\in\C^*, k_i\in\Z\}$$
and
$$R^*_{a'}=\{c\prod_{i=1}^m(t-a'_i)^{k_i}\,|\, c\in\C^*,
k_i\in\Z\}.$$ Using the isomorphism $\phi$ we see that $\Z^n\cong
R_a^*/\C^* \cong R_{a'}^*/\C^*\cong \Z^m$, yielding  $m=n$.

Since $\phi:R_a\to R_{a'}$ is an isomorphism of $\mathbb C$-algebras and $R_b\subset \mathbb C(t)$ for $b=a,a'$, by the universal mapping property of $\mathbb C(t)$, there exists a unique field automorphism in the Galois group $\widehat\phi\in \text{Gal}(\mathbb C(t)/\mathbb C)$ such that $\widehat\phi|_{R_a}=\phi$.  Now it is known that every element in this Galois group   is a M\"obius transformation, so
$$
\phi(t)=\frac{ct+d}{at+b}
$$
for some $a,b,c,d\in\mathbb C$, $ad-bc\neq 0$.

We also know that $\phi(R^*_a)=R_{a'}^*$ so that there exists $c_i\in\C^*$, $k_{ik}\in\{0,\pm 1\}$, such that
$$
\varphi(t-a_i)=c_i(t-a'_1)^{k_{i1}}\cdots (t-a'_n)^{k_{in}},\forall\,\,
i=1,2,\cdots, n
$$
where there are at most two nontrivial factors since
\begin{equation}\label{eqn1}
\varphi(t-a_i)=\frac{ct+d}{at+b}-a_i=\frac{(c-aa_i)t+d-ba_i}{at+b}.
\end{equation}
Since $\phi(t-a_i)$ is not a constant we must have that $(c-aa_i)t+d-ba_i$ and $at+b$ are relatively prime.
\vskip 5pt {\bf Case 1.}  Suppose $a=0$.  Then we may assume $b=1$ and $\phi(t)=ct+d$.  Then fixing $1\leq i\leq n$ we have
$$
ct+d-a_i=\phi(t-a_i)=c_i(t-a_j')=c_it-c_ia_j'
$$
for some $j$.   After reindexing the $a'_j$'s if necessary we may assume that $j=i$.  Thus $c_i=c$ and $d-a_i=-c_ia_i'=-ca_i'$.  Then
$$
a_i-ca_i'=d=a_k-ca_k'
$$
for any $i$ and $k$.   This gives us (a).

\vskip 5pt {\bf Case 2.} Suppose $a\neq 0$.  Note that from \eqnref{eqn1} we see that
$$
\phi(t-a_i)=\frac{c_if_i(t)}{t-a_{l_i}'}
$$
with $f_i(t)$ being either $1$ or a linear factor of the form $t-a_k'$ for some $k\neq l_ i$.
Since $a\neq 0$ we may assume $a=1$ and we have
 $$
 \frac{(c-a_i)t+d-ba_i}{t+b}=\frac{ct+d}{t+b}-a_i=\phi(t-a_i)=\frac{c_if_i(t)}{t-a_{l_i}'}
$$
for all $i$ where $f_i(t)$ is either 1 or a linear factor of the form $t-a_k$ for some $k$.    So all of the $\phi(t-a_i)$ have a fixed denominator $t+b=t-a_s'$ for some fixed $s$. As the quotients $(t-a_1)^{-1},\dots,(t-a_n)^{-1}$ are linearly independent, and the quotients $(t-a_1')^{-1},\dots, (t- a_{s-1}')^{-1},  (t- a_{s+1}')^{-1}\dots, (t-a_n')^{-1}$ are linearly independent, there must be exactly one index $k$ for which $f_k(t)=1$.    After reindexing the $a_i$'s we may assume that $k=1$.
After reindexing the $a_j'$'s we may assume $b=-a_1'$ and
$$
\phi(t-a_i)=\frac{c_i(t-a'_i)}{(t-a'_1)},\quad i\neq 1.
$$
Now setting $c=c_1$ we get
$$
(a_1-a_i)=\phi((t-a_i)-(t-a_1))=\frac{c_i(t-a'_i)}{(t-a'_1)}-\frac{c}{t-a'_1}=\frac{c_it-c_ia'_i-c}{t-a'_1}$$ and thus we see that
$$
a_1-a_i=c_i,\,\,c=(a_1-a_i)(a'_1-a'_i),\,\forall\,\, i\geq 2.$$
Now
$$
\varphi(t-a_1)=\frac{c}{t-a'_1},\,\,\text{ and
}$$¡¡$$\varphi(t-a_i)=\frac{(a_1-a_i)(t-a'_i)}{t-a'_1} \quad\text{ for all
$i\geq2$}.
$$ It is easy to see that this $\phi$ defines an isomorphism
from $R_a$ to $R_{a'}$.  Consequently,
\begin{align*}
\phi((t-a_1)^k )&=\frac{c^{k} }{(t-a'_1)^{k}}; \\
 \phi((t-a_i)^k )&=\frac{(a_1-a_i)^k(t-a'_i)^k}{(t-a'_1)^{k}}, \,\,\forall\,\,i>1.
 \end{align*}
This completes the proof.
\end{proof}

For convenience, we call isomorphisms defined in (a) and (b)
{\bf isomorphisms of the first kind} and {\bf isomorphisms of the
second kind} respectively.

\begin{corollary}\label{phit} Suppose $\{a_1, a_2, \cdots, a_n\}$ and $\{a'_1, a'_2, \cdots,
a'_n\}$ are two sets of distinct complex numbers. Then a map $\phi:
R_a\to R_{a'}$ is an isomorphism of associative algebras if and only
if its unique extension $\phi\in\Aut(\C(t))$ maps the unordered set $\{\infty,
a'_1,a'_2,\cdots,a'_n\}$ onto $\{\infty, a_1,a_2,\cdots,a_n\}$.
Moreover this isomorphism $\phi$ corresponds to a unique map from $\{\infty,
a'_1,a'_2,\cdots,a'_n\}$ onto $\{\infty, a_1,a_2,\cdots,a_n\}$. Consequently, any
such $\phi$ is completely determined by the image of $t$.
\end{corollary}

\begin{remark} Because of \Corref{phit}, automorphisms of the second kind $\phi(t)$ in part (b) of \thmref{ringisothm} are completely determined by the formula
$$
\varphi(t)=\frac c{t-a_{\tau(i_0)}}+a_{i_0},
$$
for some fixed $c$ and permutation $\tau$  (see \secref{examples} for use of this fact).
\end{remark}

Now we can prove the main result in this section.

\begin{theorem}\label{liealgiso}
Suppose $\{a_1, a_2, \cdots, a_n\}$ and $\{a'_1, a'_2, \cdots,
a'_m\}$ are two sets of distinct complex numbers and
$\psi:\mathscr{V}_a\to \mathscr{V}_{a'}$ is an isomorphism of Lie
algebras. Then $m=n$ and one of the following two cases holds
\begin{enumerate}
\item[(a).] There exists some constant nonzero $c\in\C$ such that
$a_i-a_1=c(a'_i-a'_1)$ and
$\psi((t-a_i)^{k}\partial)=c^{k-1}(t-a'_i)^{k}\partial$ for all
$k\in\Z$ and $i=1,2,\cdots ,n$ after reordering $a'_i$ if necessary.
\item[(b).] There exists some constant nonzero $c\in\C$ such that
$(a_i-a_1)(a'_i-a'_1)=c$ for all $i\ne1$, and
$$\psi((t-a_1)^k\partial)=\frac{-c^{k-1}\partial}{(t-a'_1)^{k-2}},$$
$$\psi((t-a_i)^k\partial)=\frac{-(a_1-a_i)^k(t-a'_i)^k\partial}{c(t-a'_1)^{k-2}}, \,\,\forall\,\,i>1,$$
for all $k\in\Z$ after reordering $a_i$ and $a'_i$ if necessary.
\end{enumerate}
\end{theorem}

\begin{proof}  As in \cite{MR966871}, let
$$J_1=\{D(x)\,|\, D\in \mathscr{V}_a, x\in R_a\},$$
$$J_2=\{D_1(x)D_2(y)-D_1(y)D_2(x)\,|\, D_1, D_2\in \mathscr{V}_a, x,
y\in R_a\}.$$ Then $J_1=6R_a=R_a$ and $J_2=0$.
This means that we can apply \cite[Theorem 2]{MR966871}, to see that there exists an associative
algebra isomorphism $\phi: R_a\to R_{a'}$ such that
$\psi(x\partial)=\phi(x)(\phi\circ\partial\circ \phi^{-1})$ for all
$x\in R_a$. From \thmref{ringisothm} we know that $\phi$ has two different
kinds.

\vskip 5pt {\bf Case 1.} $\phi$ is of the type in (a) of \thmref{ringisothm}.

We have $\phi\circ \partial\circ \phi^{-1}(t)=\phi\circ
\partial(c^{-1}(t-a_1)+a'_1)=c^{-1},$
i.e., $\phi\circ \partial\circ \phi^{-1}=c^{-1}\partial$.
Consequently $\psi((t-a_i)^k\partial)=c^{k-1}(t-a'_i)^k\partial$.

\vskip 5pt {\bf Case 2.} $\phi$ is of the type in (b) of \thmref{ringisothm}.

 Now we compute
$$\phi\circ \partial\circ \phi^{-1}(t)=\phi\circ
\partial(\frac{c}{t-a_1}+a'_1)$$
$$=\phi(\frac{-c}{(t-a_1)^2})=-c^{-1}(t-a'_1)^2,$$
i.e., $\phi\circ \partial\circ \phi^{-1}=-c^{-1}(t-a'_1)^2\partial$.
Consequently,
\begin{align*}
\psi((t-a_1)^k\partial)&=\frac{-c^{k-1}\partial}{(t-a'_1)^{k-2}}; \\
\psi((t-a_i)^k\partial)&=\frac{-(a_1-a_i)^k(t-a'_i)^k\partial}{c(t-a'_1)^{k-2}}, \,\,\forall\,\,i>1.
\end{align*}
This completes the proof.
\end{proof}

Similarly to what we defined earlier, we call isomorphisms in (a) and (b)
{\bf isomorphisms of the first kind} and {\bf isomorphisms of the
second kind} respectively. Now we can determine isomorphism classes
of the algebra $\W_a$. 
For any $a=(a_1,\cdots ,a_n)\in\C^n$ with $a_1,\cdots ,a_n$
distinct, we define
$$\mathfrak{C}_1=\{(xa_1+y, xa_2+y,\cdots ,xa_n+y)\,|\,x\in\C^*, y\in\C\}$$
and
$$\mathfrak{C}_2=\{(x(a_1-a_i)^{-1}+y,\cdots ,x(a_n-a_i)^{-1}+y)\,|\,x\in\C^*,
y\in\C, 1\leq i\leq n\},$$ where we set $0^{-1}=0$.
By direct computations we can easily prove the following theorem.

\begin{theorem}\label{classificationthm} Suppose that $\{a_1, a_2, \cdots, a_n\}$ are distinct complex
numbers and $a=(a_1,\cdots ,a_n)$.

\begin{enumerate}
\item[(a).] The $\mathfrak{C}_1$ is the set of all $a'\in \C^m$ for some
$m\in\N$ such that $\W_{a'}\simeq\W_a$ under  isomorphisms of the
first kind;

\item[(b).]  $\mathfrak{C}_2$ is  the set of all $a'\in \C^m$ for some
$m\in\N$ such that $\W_{a'}\simeq\W_a$  under isomorphisms of the
second kind;

\item[(c).]  $\mathfrak{C}_1\cup\mathfrak{C}_2$ is the set of all $a'\in
\C^m$ for some $m\in\N$ such that $\W_{a'}\simeq\W_a$.
\end{enumerate}

\end{theorem}

Denote by $\Aut(\W_a)$ the automorphism group of $\W_a$. Next we
will study $\Aut(\W_a)$ for various $a=(a_1, a_2, \cdots, a_n)$. Let
us start with some simple examples.

\vskip 5pt {\bf Example 1.} If $n=1$, it is well-known that $\W_a$
is isomorphic to the centerless Virasoro algebra for any $a\in\C$,
whose automorphism group is $\C^*\times \Z/(2\Z)$.

\begin{remark} From \thmref{liealgiso} we know that if $n>1$,
each automorphism of $\W_a$ of the second kind corresponds to a
permutation on the set $\{a_1, a_2, \cdots, a_n\}$ and a choice of
$a_1$, while each automorphism  of $\W_a$ of the first kind
corresponds to a permutation on the set $\{a_1, a_2, \cdots, a_n\}$
(independent of the choice of $a_1$).     \end{remark}

\vskip 5pt {\bf Example 2.}
 If $n=2$, any two $\W_a$ and $\W_{a'}$
are isomorphic for all $a,a'\in\C^2$ provided $a_1\neq a_2$ and
$a'_1\neq a'_2$. Without loss of generality, we may assume that
$a=(0,1)$. By \thmref{liealgiso}, we can easily deduce that the first kind
automorphisms of $\R_a$ are $\iota,\phi$ where
$$\iota={\rm identity\,\, map\,\,\,and\,\,\,} \phi(t)=1-t;$$ the second kind automorphism of $\W_a$ consists of
$\pi_1,\pi_2,\pi_3,\pi_4$:
$$\pi_1(t)=\frac{1}{t}; \,\,\, \pi_2(t)=\frac{1}{1-t};$$
$$\pi_3(t)=1-\frac{1}{1-t}; \,\,\, \pi_4(t)=1- \frac{1}{t}.$$ It is easy to see that
$$\phi^2=\pi_2^3=\iota,\,\,\phi\pi_2\phi=\pi_2^2,\,\, \pi_2^2=\pi_4,
\phi\pi_2=\pi_1,\,\,\phi\pi_4=\pi_3.$$ Then we have that
$\Aut(\W_a)\cong D_3$, the $3$-rd dihedral group.

\vskip 5pt To study $\Aut(\W_a)$ in general,  from \thmref{classificationthm} we can
assume if necessary that
\begin{equation}\label{3.1}
a_1+a_2+\cdots+a_n=0
\end{equation}
Let $\s$ be an automorphism  of $\W_a$ of the first kind
corresponding to a permutation which we still denote by $\s$ on the
index set $\{1, 2, \cdots, n\}$, i.e.,
$a_{\s(i)}-a_{\s(j)}=c(a_i-a_j)$ for all $i,j$, where $c\in\C^*$. We
can use the same symbol $\s$ for different meanings since $\s=1$ on
$\W_a$ iff $\s=1$ on $\{1, 2, \cdots, n\}$.

For convenience, we let
 $S_a$ be the set of all triples $(i,i',c)$ such that
\begin{equation}\label{3.2}
\{c(a_j-a_i)^{-1}\,|\,j\ne i\}=\{(a_j-a_{i'})\,|\,j\ne i'\}
\end{equation}
 for
some constant $c$ depending on $i,i'$.

\begin{theorem}\label{2ndisothm} Let  $a_1, a_2, \cdots, a_n$ be distinct complex
numbers with $n\ge 3$, and let $a=(a_1,\cdots ,a_n)$ and
$\s\in\Aut(\W_a)$.
\begin{enumerate}
\item[(a).] Suppose \eqnref{3.1} holds. Then
$\s$ is of the first kind iff there exist distinct complex numbers
$b_1,b_2,\cdots, b_r$ and $A_d=\{\omega^i\,|\,i=0,1,\cdots, d-1\}$
where $\omega$ is a $d$-th primitive root of unity such that $\{a_1,
a_2, \cdots, a_n\}=\cup_{i=1}^rA_db_i$. For each such partition of
$\{a_1, a_2, \cdots, a_n\}$ we get $d$ automorphisms of the first kind.
\item[(b).] $\s$ is of the second kind iff  there exist a pair $(i,i')$ and $c\in\C^*$ such that
 \eqnref{3.2} holds.
Furthermore, the set of automorphisms of the second kind and the set
$S_a$ is in 1-1 correspondence.
\end{enumerate}
\end{theorem}

\begin{proof}  From \cite[\thmref{liealgiso}]{MR966871}
we know that there exists   an  associative algebra automorphism
$\phi$ of $R_a$ such that $\s(f(t)\p)=\phi(f)(\phi\circ\partial\circ
\phi^{-1})$ for all $f(t)\in R_a$.

\vskip 5pt {\bf (a).} ``$\Rightarrow$". In this case, there exists
$c\in\C^*$ such that $a_{\sigma(i)}-a_{\sigma(j)}=c(a_i-a_j)$ for
any $1\leq i,j\leq n$. We may assume that $\s\ne1$. Let  $O$ be an
orbit of $\sigma$  in $\{1,2,\cdots ,n\}$. Suppose $d=|O|>1$. Without
loss of generality we may assume that $O=\{1,2,\cdots ,d\}$ and
$s=\s^{s-1}(1)$ for all $s=1,\cdots ,d$. Then we have
$$a_2-a_1=c(a_{1}-a_d)=c^2(a_{d}-a_{d-1})=\cdots =c^d(a_{2}-a_{1}),$$
yielding $c^d=1$. From
$$a_{s}-a_{1}=(a_2-a_1)(1+c+c^{2}+\cdots +c^{s-2}),\,\,\forall\, 2\leq s\leq d,$$
 we see that $c^{s-1}\ne1$ for any $1\le s\le d$ or $c= 1$.  We see that $c\neq 1$ because otherwise
 if we set
 $$
 x=a_2-a_1=a_3-a_2=\cdots =a_1-a_d
 $$
 we get $a_k=a_1+(k-1)x$ and hence $x=a_1-a_d=a_1-(a_1-(d-1)x)=(1-d)x$ or $d=0$.

Thus, $c$ is a primitive $d$-th root of unity.

Then we deduce that  $\{1,2,\cdots ,n\}$  is a union of all
$\sigma$-orbits and each orbit has the  length $1$ or $d$. If
$O'=\{a_j\}$ is a singleton orbit of $\s$, then we have
$a_{\s^{s+1}(1)}-a_j=c(a_{\s^s(1)}-a_{j})$, i.e.,
$ca_{\s^s(1)}-a_{\s^{s+1}(1)} =(c-1)a_{j}$. Adding them up from
$s=0$ to $d-1$, we can deduce
$$d(c-1)a_{j}=\sum_{s=0}^{d-1}(ca_{\s^{s+1}(1)}-a_{\s^s(1)})=(c-1)\sum_{s=0}^{d-1}a_{\s^s(1)},$$
which gives $a_j=\sum_{s=1}^{d}a_s/d$.

It follows from this that if there is a singleton orbit for
$\sigma$, then the singleton orbit is unique, and the sum
$\sum_{s\in O}a_{s}$ is independent of $O$ for  any orbit $O$ of
$\sigma$ with $|O|>1$.

If $O'$ is another non-singleton orbit of $\sigma$ containing $a_i$,
then
$a_{\s^{s+1}(i)}-a_{\s^{s+1}(1)}=c(a_{\s^{s}(i)}-a_{\s^{s}(1)})$,
i.e., $ca_{\s^s(1)}-a_{\s^{s+1}(1)} =ca_{\s^s(i)}-a_{\s^{s+1}(i)}$.
Adding them up from $s=0$ to $d-1$, we can deduce
$(c-1)\sum_{s=0}^{d-1}a_{\s^s(i)}=(c-1)\sum_{s=0}^{d-1}a_{\s^s(1)}$,
which yields
$\sum_{s=0}^{d-1}a_{\s^s(i)}=\sum_{s=0}^{d-1}a_{\s^s(1)}$. Again,
the sum $\sum_{s\in O'}a_{s}$ is independent of $O'$ for  any orbit
$O'$ of $\sigma$ with $|O'|>1$.  From the assumption \eqnref{3.1} we see
that $\sum_{s\in O'}a_{s}=0$ for any orbit $O'$.

Now returning to $O$, we have
$$a_3-a_2=c(a_2-a_1), ..., a_d-a_2=c(a_{d-1}-a_1), a_{1}-a_2=c(a_d-a_1).$$
Adding them up and using $\sum_{s=1}^{d}a_s=0$, we deduce that
$a_2=ca_1$. Similarly, $a_{s+1}=ca_s$ for all $1\leq s\leq d-1$, and
we conclude that $O=a_iA_d$ for any $1\leq i\leq d$. Thus
$O'=a_iA_d$ for any orbit $O'$ which can be a singleton consisting
of $0$ and any $a_i\in O'$. So this direction of (a) is proved.

``$\Leftarrow$". This direction can be verified by direct
computation by taking $a_{\s(i)}=\omega a_{i}$.

\vskip 5pt  {\bf (b).} ``$\Rightarrow$". From \thmref{liealgiso}, we know
that there exist $1\leq i\leq n$ and a permutation $\tau$ on the
index set $\{1,2,\cdots,n\}$ such that

$$\phi(t-a_1)=\frac{c}{t-a_{\tau(1)}}, {\rm
\,\,and\,\,\,} (a_j-a_i)(a_{\tau(j)}-a_{\tau(i)})=c,\, \forall\,
j\neq1.$$ Thus \eqnref{3.2} holds.

``$\Leftarrow$". This direction can be verified by direct
computation.
\end{proof}
It is not hard to see from the proof of this theorem the following
\begin{corollary}\label{cyclic} Suppose\eqnref{3.1} holds.  The set of automorphisms of the first kind forms a cyclic subgroup generated by the rotation $z\mapsto e^{\imath \theta}z$
where $\theta=2\pi/d$ for some positive integer $d$.
\end{corollary}

From the above corollary and \thmref{2ndisothm} we know that if $n>1$,
there are at most $n$ automorphism of $\W_a$ of the first kind, and
there are at most $n^2(n-1)$  automorphisms of $\W_a$ of the second kind.
So the group $\Aut(\W_a)$ is finite with order at most $n^3-n^2+n$.

{\bf Example 3.} If $n=3$, we may assume that $a=(0,1,x)$ for some
$x\in\C\setminus\{0,1\}$. Applying \thmref{liealgiso} we deduce that
$\W_{(0,1,x)}\simeq \W_{(0,1, y)}$ if and only if one of the
following holds:
\begin{equation}\label{3.3}
y=x, \,\, 1-x,\,\, \frac1x,\,\,  1-\frac1x,\,\,
\frac1{1-x},\,\,  1-\frac1{1-x}.
\end{equation}

The following are clearly nontrivial automorphisms  of
$\W_{(0,1,x)}$ of  the second kind:
$$\s_1(t)=\frac xt$$
which is obtained by taking $a=(0,1,x)$ and $a'=(0,x,1)$,
$$\s_2(t)=\frac {t-x}{t-1}$$
which is obtained by taking $a=(1,0,x)$ and $a'=(1,x, 0)$,
$$\s_3(t)=\frac {x(t-1)}{t-x}$$
which is obtained by taking $a=(x,0,1)$ and $a'=(x,1,0)$. It is easy
to see that  $H=\{\iota, \s_1,\s_2,\s_3\}$ is a subgroup of
$\Aut(\W_a)$ isomorphic to the Klein four group $V=D_2$.

\vskip 5pt {\bf Case 1:} All of the the six values in \eqnref{3.3} are
pairwise distinct.

From \thmref{2ndisothm} we know that the identity map is the only
automorphism of $\W_{(0,1,x)}$ of  the first kind.  In this case,
$\Aut(\W_a)=H=V=D_2$.

If two of the values in \eqnref{3.3} are the same, then $x=-1, 2, \frac 12$
or $x=\frac{1\pm\sqrt{-3}}{2}$.

\vskip 5pt  {\bf Case 2:} $x=-1, 2, \frac 12$.

Since all the three corresponding algebras are isomorphic, we may
take $x=-1$. Then we have the automorphism of the first kind
$\tau_1(t)=-t$. Thus $\Aut(\W_a)=H\cup H\tau_1\cong D_4=\langle a,b\,|\, a^4=b^2=1,\enspace bab=a^{-1}\rangle$ where for example $a=\sigma_2\tau$ and $b=\tau$, in this case.

\vskip 5pt {\bf Case 3:} $x=\frac{1\pm\sqrt{-3}}{2}$.

Since all the two corresponding algebras are isomorphic, we may take
$x=\frac{1+\sqrt{-3}}{2}$. Then we have the automorphism of the
first kind $\tau_2(t)=-x(t-1)$. It is clear that $\tau_2^3=\iota$.
Thus $\Aut(\W_a)=H\cup H\tau_2\cup H\tau_2^2\cong A_4$ the alternating group on four letters in this case.
This is because $\Aut(\W_a)$ is nonabelian and contains $H$  as a normal subgroup.   $D_6$ contains an element of order $6$ whereas $\Aut(\W_a)$ does not.

Therefore we have determined the automorphism group for all
$\W_{(0,1,x)}$. This agrees with the remark at the end of Sect.1 in
\cite{MR1261553}. 

\vskip 5pt {\bf Example 4.} It is straightforward to see that
$\Aut(R_{(1,-1,2,-2)})=C_2=D_1$ (where $c=-1$, $i_0=1$,
$\sigma=(1,2)(3,4)$), $\Aut(R_{(1, \omega, \omega^2, 10, 10\omega,
10\omega^2)})=C_3$ (where $c=\omega^2$, $i_0=1$,
$\sigma=(1,2,3)(4,5,6)$), where $\omega$ is a primitive cubic root
of the unity.  From \thmref{liealgiso} we know that the  the only
automorphism of $\W_{(0,-1,-2, 3)}$ is the identity.

\begin{theorem}[\cite{MR1427489} Theorem 2.3.1] Any finite subgroup of the automorphism group $\text{Aut}(\hat{\mathbb C})$ of the Riemann sphere $\hat{\mathbb C}=\mathbb C\cup \{\infty\}$ is conjugate to a rotation group.
\end{theorem}

\begin{theorem}[\cite{MR1427489} Theorem 2.6.1]\label{Klein} If $G$ is a finite rotation group of the group of automorphisms of the Riemann sphere, then it is isomorphic to one of the following:
\begin{enumerate}[(a).]
\item  A cyclic group $C_n=\{s|s^n=1\}$ for $n\geq 1$.
\item  A dihedral group $D_n=\langle s,t\,|\, s^n=1=t^2,\enspace tst=s^{-1}\rangle$.
\item  A platonic rotation group $A_4,S_4$ or $A_5$.
\end{enumerate}
\end{theorem}

\noindent
\begin{remark} Despite having proved \thmref{liealgiso} and \thmref{2ndisothm}, we can see from the the above examples that it is
{\bf not} easy to completely determine $\Aut(\W_a)$ for an arbitrary $a$, but the above theorem tells us that $\text{Aut}\,(\W_{a})$ is one of the groups of type (a)-(c) above.
\end{remark}

In the appendix at the end of the paper we show that each of the
finite groups listed in \thmref{Klein} appear as an $\Aut(R_a)$ for
some $a$.

Here is one last conclusion we can get about $\W_a$:
\begin{theorem}
Suppose $\{a_1, a_2, \cdots, a_n\}$ is a set of distinct complex
numbers. Then all derivations of  $\mathscr{V}_a$ are inner
derivations.
\end{theorem}

\begin{proof} This is a direct corollary of \thmref{liealgiso} in \cite{MR966871}.
\end{proof}
%

\section{Universal central extensions of $\mathscr{V}_a$}

In this section we shall determine the universal central extension
of the Lie algebra $\W_a$.

For any $1\leq i\leq n$,we define the skew-symmetric bilinear map
$\phi_i: \,\W_a\times\W_a\rightarrow \C$ as follows on the basis
elements in \lemref{lemma1}:

\begin{equation}\label{4.1}
\phi_i(t^{k+1}\partial,
(t-a_i)^{-l+1}\partial)=\begin{cases} {k+1\choose
l+1}a_i^{k-l}(l^3-l),\,\,\forall\, k,l\in\Z_+,\hskip 5pt\text{ if $a_i\neq 0$,}  \\
\delta_{k,l}(l^3-l), \,\,\forall\,
k,l\in\Z\hskip 50pt\text{ if $a_i=0$,}
\end{cases}
\end{equation}
 \begin{align*}
\phi_i((t-a_i)^{-k}\partial, (t-a_j)^{-l}\partial)
&=\frac{(k+l+1)!}{(k-1)!(l-1)!(a_j-a_i)^{k+1}(a_i-a_j)^{l+1}},\\
&\hskip 100pt\forall \,
k,l\in\N, j\neq i,
\end{align*}
and $\phi_i=0$ at all other pairs of
basis elements. From this it is easy to see that
\begin{align}\phi_i((t-a_i)^{k+1}\partial,
(t-a_i)^{-l+1}\partial)=\delta_{k,l}(l^3-l), \,\,\forall\,
k,l\in\Z.
\end{align}
which is the motivation for the formula \eqnref{4.1}.

Now we have the following

\begin{lemma} The above defined $\phi_i$ are $2$-cocycles on $\W_a$ which are not $2$-coboundaries for all $i=1,...,n$.
\end{lemma}
\begin{proof}  One can prove by direct computation using \Corref{combinatorialid} that the $\phi_i$ are 2-cocycles.   Another approach is the following:  Let $\Sigma$ be a Riemann surface.  Schlichenmaier has shown in \cite{Schl1} that
$$
\gamma_{C,R}(e,h)=\frac{1}{24\pi}\int_C\left(\frac{1}{2}(f'''g-gf''')-R(f'g-fg')\right) \,dz.
$$
is a $2$-cocyle where $C$ is any cycle on $\Sigma$,  $R$ is a projective connection (see \cite[Equation 3.13]{Schl1}) and  $e$ and $h$ are vector fields that are locally represented by the form
$$
e_|(z)=f(z)\frac{\partial}{\partial z},\quad h_|(z)=g(z)\frac{\partial}{\partial z}.
$$
   In our case since the genus is zero, using the coordinate $z$, one can set $R=0$.   If we take $C=C_i$ to be a loop around the point $a_i$, but not inclosing the other points $a_j$, then the above integral gives us the $\gamma_{C_i,0}=\phi_i$ above.  We omit the detailed computation.
\end{proof}

  In \cite{MR1682304} the Wagemann describes the continuous cohomology groups of meromorphic vector fields on a Riemann surface $\Sigma $ with poles at a finite number of points.  In this work the modules and algebras under consideration are topological objects and the cohomology is computed using continuous cochains whereas below we (and Skryabin) deal with arbitrary linear cocycles.

 Suppose $R$ is an associative commutative ring with unit, $W\subset \text{Der}R$ is a Lie algebra over $Z$ such that $W$ is a rank 1 projective $R$-module such that
 \begin{enumerate}
 \item 2 is invertible in $R$ and $3R=R$,
 \item $\Omega^1=\text{Hom}_R(W,R)=R\cdot dR$.
 \end{enumerate}
 Then one of the main results of Skryabin \cite{MR2035385} is that the kernel of the universal central extension of $W$ is isomorphic to $H^1(\Omega)$ where $H^1(\Omega)$ is the de Rham cohomology with respect to $W$.  His proof is rather indirect and somewhat involved.  In our setting we offer below a more direct proof  of the determination of the $2$-cohomology group of $\W_a$.

\begin{theorem}\label{uce} $H^2(\mathscr{V}_a, \C)=\sum_{i=1}^n\C \bar{\phi_i},$
where $\bar{\phi_i}$ is the image of $\phi_i$ in $H^2(\mathscr{V}_a,
\C)$ for any $i=1,\cdots, n$.
\end{theorem}

\begin{proof}
Suppose that $\psi: \W_a\times\W_a\rightarrow \C$ is an arbitrary
$2$-cocycle. 
%
%
For any $1\leq i\leq n$ and $k\in\Z$, we define a linear map by
action on basis elements in \lemref{lemma1} as follows
$$\chi_{i,k}:\,\,\W_a\,\rightarrow\,\C,\,\,\,\chi_{i,k}((t-a_j)^{l}\p)
=\delta_{i,j}\delta_{k,l},\,\,\forall\,1\leq j\leq n,\,l\in-\Z_+,$$
$$\chi_{i,k}(t^{l}\p)=\delta_{k,l},\,\,\forall\,1\leq j\leq n,\,l\in\N.$$
Each $\chi_{i,k}$ induces a $2$-coboundary $\psi_{i,k}$ via the
formula
$$\psi_{i,k}:\,\,\W_a\times\W_a\,\rightarrow\,\C,\,\,\,\,\psi_{i,k}(f\p,g\p)
=\chi_{i,k}([f\p,g\p]),\,\,\forall\,f\p,g\p\in\W_a.$$ It is clear
that any infinite sum $\sum_{j\in\Z}y_{i,j}\psi_{i,j}$ is a
well-defined\break  $2$-coboundary for any $i$ and any
$y_{i,j}\in\C$ since it is induced from a linear function
$\sum_{j\in\Z}y_{i,j}\chi_{i,j}$ on $\W_a$.

 Recall that for any $1\leq i\leq n$, we have a
centerless Virasoro subalgebra
$\Vir^{(i)}=\oplus_{k\in\Z}\C(t-a_i)^k\p$. We see that the
restriction of $\phi_i$ on $\Vir^{(i)}\times\Vir^{(i)}$ is a
nontrivial $2$-cocycle on $\Vir^{(i)}$. Then by the cohomology
theory of Virasoro algebra, there exist $x_i, y_{i,j}\in\C$ for
$1\leq i\leq n$ and $j\in\Z$ such that
\begin{equation}\label{4.16}
\psi|_{\Vir^{(i)}\times\Vir^{(i)}}=x_i\phi_i|_{\Vir^{(i)}\times\Vir^{(i)}}
+\sum_{j\in\Z}y_{i,j}\psi_{i,j}|_{\Vir^{(i)}\times\Vir^{(i)}}.
\end{equation}

We replace $\psi$ with
$$\psi-\sum_{i=1}^nx_i\phi_i-\sum_{j\in\Z}y_{1,j}\psi_{1,j}-\sum_{i=2}^n\sum_{j\in\N}y_{i,-j}\psi_{i,-j}.$$
Then we obtain that
%
%
%
\begin{equation}
\psi((t-a_1)^{k}\partial,
(t-a_1)^{l}\partial)=0,\,\,\forall\, k,l\in\Z,\label{4.17}
\end{equation}
\begin{equation}
\psi((t-a_i)^{k}\partial,
(t-a_i)^{l}\partial)=0,\,\,\forall\, 2\leq i\leq n,\,
k+l\leq0{\rm\,\,or\,\,}k+l=2.\label{4.18}
\end{equation}
Now we need only to prove
that $\psi=0$, i.e.,
\begin{equation}
\psi((t-a_i)^{k}\partial,
(t-a_i)^{l}\partial)=0\,\,\forall\, 2\leq i\leq n,\,
k+l\geq3,{\rm\,\,or\,\,}k+l=1,\label{4.19}
\end{equation}
\begin{equation}
\psi((t-a_i)^{k}\partial,
(t-a_j)^{l}\partial)=0,\,\,\forall\, k,l\in\Z, i\ne j.\label{4.20}
\end{equation}

%
\vspace{0cm}

{\bf Step 1.} \eqnref{4.19} holds.

\eqnref{4.19} holds for $k,l\in\Z_+$ because of \eqnref{4.17}. Now we need only to
prove
\begin{equation}
\psi((t-a_i)^{k}\partial, (t-a_i)^{s-k}\partial)=0,\,\,
\forall\, k> s,\,\,{\rm and \,\,}s\geq 3 {\rm\,\,or\,\,}
s=1.\label{4.21}
\end{equation}

 For any $1\leq i\leq n$ and $k\geq s\geq 3$, we have that
\begin{align}
\psi((k-2)(t-a_i)^{k+1}&\partial,
(t-a_i)^{s-k-1}\partial) \notag \\
 &=\psi([(t-a_i)^{2}\p,(t-a_i)^k\p],
(t-a_i)^{s-k-1}\partial) \notag \\
&=\psi([(t-a_i)^{2}\p,(t-a_i)^{s-k-1}\partial],(t-a_i)^k\p)  \notag \\
&\quad +\psi((t-a_i)^{2}\p,[(t-a_i)^k\p,(t-a_i)^{s-k-1}\partial])\notag  \\
&=\psi((s-k-3)(t-a_i)^{s-k}\p,(t-a_i)^k\p) \notag \\
&\quad  +\psi((t-a_i)^{2}\p,(s-2k-1)(t-a_i)^{s-2}\partial)\notag \\
&=(k+3-s)\psi((t-a_i)^{k}\p,(t-a_i)^{s-k}\p).\label{4.22}
\end{align}
From $\psi((t-a_i)^{s}\partial, (t-a_i)^{0}\partial)=0$ for all
$s\geq 3$, and applying induction on $k$ to \eqnref{4.22} we deduce that
\begin{equation}
\psi((t-a_i)^{k}\partial, (t-a_i)^{s-k}\partial)=0,\,\, \forall\,
k\geq s\geq 3.\label{4.23}
\end{equation}

Now we consider \eqnref{4.21} for $s=1$. For any $k\in\N$ we have
\begin{align*}
0&=-(2k+1)\psi((t-a_i)\p,\p) \\
&=\psi((t-a_i)\p,[(t-a_i)^{k+1}\p,(t-a_i)^{-k}\p]) \\
&=\psi([(t-a_i)\p,(t-a_i)^{k+1}\p],(t-a_i)^{-k}\p) \\
&\quad  +\psi((t-a_i)^{k+1}\p,[(t-a_i)\p,(t-a_i)^{-k}\p]) \\
&=k\psi((t-a_i)^{k+1}\p,(t-a_i)^{-k}\p) \\
&\quad  -(k+1)\psi((t-a_i)^{k+1}\p,(t-a_i)^{-k}\p)\\
&=-\psi((t-a_i)^{k+1}\p,(t-a_i)^{-k}\p),
\end{align*}
which is exactly  \eqnref{4.21} for $k\ge s=1$. This completes Step 1.

\vskip 5pt {\bf Step 2.} \eqnref{4.20} holds.

From \eqnref{4.17}, \eqnref{4.18} and \eqnref{4.19}, we know that
\begin{equation}
\psi(f(t)\p,
g(t)\p)=0,\,\,\forall\, f\in R_a, \,g\in\C[t].\label{4.24}
\end{equation}
 Then  we
need only to prove \eqnref{4.20} for $k,l\in-\N$.

Fix any $i, j$ with $1\leq i\neq j\leq n$. Let
$c_{kl}=\psi((t-a_i)^{-k}\p, (t-a_j)^{-l})$ for $k,l\in\N$. Then for
any $l,m,k\in\N$ with $l\geq m$, using \eqnref{4.24} we obtain
\begin{align}
\psi((m+k)&(t-a_i)^{m-k-1}\p, (t-a_j)^{-l}\p)  \notag \\
&=\psi([(t-a_i)^{-k}\p, (t-a_i)^{m}\p], (t-a_j)^{-l}\p) \notag \\
&=\psi([(t-a_i)^{-k}\p, (t-a_j)^{-l}\p], (t-a_i)^{m}\p) \notag\\
&\quad  +\psi((t-a_i)^{-k}\p, [(t-a_i)^{m}\p, (t-a_j)^{-l}\p)] \notag\\
&=\psi((t-a_i)^{-k}\p, [(t-a_i)^{m}\p, (t-a_j)^{-l}\p]) \notag \\
&=\psi((t-a_i)^{-k}\p, -l(t-a_i)^{m}(t-a_j)^{-l-1}\p-m(t-a_i)^{m-1}(t-a_j)^{-l}\p) \notag \\
&=\psi\left((t-a_i)^{-k}\p, -l\frac{\sum_{s=0}^{m}{{m}\choose{s}}(a_j-a_i)^s(t-a_j)^{m-s}}{(t-a_j)^{l+1}}\p\right) \notag\\
&\quad +\psi\left((t-a_i)^{-k}\p, -m\frac{\sum_{s=0}^{m-1}{{m-1}\choose{s}}(a_j-a_i)^s(t-a_j)^{m-1-s}}{(t-a_j)^{l}}\p\right) \notag\\
&=-l\sum_{s=0}^{m}{{m}\choose{s}}(a_j-a_i)^sc_{k,l+s+1-m} \notag \\
&\quad -m\sum_{s=0}^{m-1}{{m-1}\choose{s}}(a_j-a_i)^sc_{k,1+s+l-m}. \label{4.25}
\end{align}
When $m=1$, formula \eqnref{4.25}
gives
\begin{equation*}
(l+k+2)c_{k,l}=l(a_i-a_j)c_{k,l+1},\,\,\forall\, l,k\geq1.\label{4.26}
\end{equation*}
By symmetry we can deduce that
\begin{equation*}
(l+k+2)c_{k,l}=k(a_j-a_i)c_{k+1,l},\,\,\forall\, l,k\geq1.\label{4.27}
\end{equation*}
Using induction on $k$ and $l$, we can get
\begin{equation}
c_{k,l}=\frac{(l+k+1)!}{3!(k-1)!(l-1)!(a_j-a_i)^{k-1}(a_i-a_j)^{l-1}}c_{1,1},\,\,\forall\, l,k\geq1.\label{4.28}
\end{equation}
Then take $l=m=3$ and $k=1$ in \eqnref{4.25}, and using \eqnref{4.24} we get
$$0=\psi(4(t-a_i)\p, (t-a_j)^{-3}\p)$$
$$=-3\sum_{s=0}^{3}{{3}\choose{s}}(a_j-a_i)^sc_{1,s+1}-3\sum_{s=0}^{2}{{2}\choose{s}}(a_j-a_i)^sc_{1,s+1}.$$
Substituting \eqnref{4.28} into the above equation, we deduce
$$-3(\sum_{s=0}^{3}{{3}\choose{s}}(-1)^s\frac{(s+3)!}{3!s!}
+\sum_{s=0}^{2}{{2}\choose{s}}(-1)^s\frac{(s+3)!}{3!s!})c_{1,1}=0.$$
Simplifying it, we have $c_{1,1}=0$.
%
Then formula \eqnref{4.28} gives $c_{k,l}=0$ for all $k,l\in\N$. Step 2
follows and the theorem is proved.
\end{proof}

By \thmref{uce}, we can get the universal central extension
$\widetilde{\W_a}$ of $\W_a$ as follows:
$\widetilde{\W_a}=\W_a\oplus_{i=1}^n\C c_i$ as a vector space with the Lie
bracket given by
$$[f\p+xc,g\p+yc]=[f\p,g\p]+\sum_{i=1}^n\phi_i(f\p,g\p)c_i,\,\,\forall f,g\in R_a, x,y\in\C.$$
Similar to the classical Virasoro algebra, it will be more proper to
call the universal central extension $\widetilde{\W_a}$ an {\it
$(n+1)$-point Virasoro algebra}.
\medskip

\begin{remark} Note that the $2$-cocycles $\phi_i,
i=1,\cdots,n$ seem not so elegant since they have nonzero values on
many pairs of elements from ``negative" part of the algebra. However,
we can construct a ``nice" $2$-cocycle $\phi=\sum_{i=1}^n\phi_i,$
which has a more ``natural" action on $\W_a$:
$$\phi((t-a_i)^{k+1}\partial,
(t-a_i)^{-l+1}\partial)=\delta_{k,l}(k^3-k),\,\,\forall\, 1\leq
i\leq n{\rm\,\,and\,\,}k,l\in\Z,$$
$$\phi((t-a_i)^{-k}\partial,
(t-a_j)^{-l}\partial)=0,\,\,\forall\, 1\leq i\neq j\leq
n{\rm\,\,and\,\,}k,l\in\N.$$
\end{remark}

\vskip 5pt This $\phi$ is the {\it separating cocycle} in \cite{MR1058993}. Using the $2$-cocycle $\phi$, we can obtain some
combinatoric identities which might be interesting. Here is an
example.

\begin{corollary} For any $k,l,m\in\N$ with $m<k+l+3$ and any $x,y\in\C$, we have
the following identity
$$\sum_{s=1}^{k+1}{k+l-s\choose k+1-s}{m\choose s+2}(2k+1-s)(s+1)^3(y-x)^sx^{m-s-2}$$
$$+\sum_{s=1}^{l+1}{l+k-s\choose l+1-s}{m\choose s+2}(2l+1-s)(s+1)^3(x-y)^sy^{m-s-2}=0.$$
\end{corollary}

\begin{proof} The above identity is trivial if $x=y$. Next we assume that $x\ne y$.
Consider some $\W_a$ with $a=(x,y)\in\C^2$. Using the $2$-cocycle
$\phi$ and computing  for all $m\in\Z_+$ and $l,k\in\N$:
$$\phi(t^{m}\p, [(t-a_i)^{-k}\p, (t-a_j)^{-l}\p])\hskip 3cm $$
$$=\phi([t^{m}\p, (t-a_i)^{-k}\p], (t-a_j)^{-l}\p)$$ $$\hskip 2cm+\phi((t-a_i)^{-k}\p, [t^{m}\p, (t-a_j)^{-l}\p]),$$
we can obtain the identity in the lemma. We omit the details.
\end{proof}

\section{Modules of Densities over $\mathscr{V}_a$}

For any $\a=(\a_1, \a_2, \cdots, \a_n)\in\C^{n}$, and $\b\in\C$, let
$V(\a,\b)=R_az$ be the free rank one $R_a$-module with generator $z$.  We will define an action $\cdot$ of  $\mathscr{V}_a$ on the $R_a$-module $V(\a,\b)$ so as to satisfy the following:
\begin{align*}&
\left((t-a_i)^k\partial\right)\cdot \left((t-a_1)^{k_1}(t-a_2)^{k_2}\cdots (t-a_n)^{k_n }z \right) \\
&\quad=\left(\sum_{j=1}^n\frac{\a_j+k_j+\delta_{i,j}k\b}{t-a_j}\right)
(t-a_1)^{k_1 }\cdots (t-a_i)^{k+k_i }\cdots
(t-a_n)^{k_n }z,
\end{align*}
 for all $i=1,\cdots ,n$ and $k\in\Z$ (here one may also have $a_i=0$ for some $i$).  Unfortunately this equation does not immediately tell us that the action is well defined as the factors $(t-a_1)^{k_1}(t-a_2)^{k_2}\cdots (t-a_n)^{k_n }z$ are not linearly independents. The formula above is motivated by the idea that $z$ should be the algebraic equivalent of the multivalued complex ``densities'' $\prod_{i=1}^n(t-a_i)^{\alpha_i}(dt)^\beta$ (see also \cite[Section 2.1]{MR1104280}, \cite[Equation 1]{MR678001}, \cite[Equation 1.8]{MR1021978} and \cite[equation 8]{MR1039524}).  To avoid
any ambiguity we will not omit the dot for the action.
Then the
above action may be rewritten as
\begin{equation}
(f\p)\cdot (gz)=f\p(gz)+\b\p(f)gz,\,\,\forall f,g\in R_a. \label{5.1}
\end{equation}
 To make this well defined we only need to make sure that $\partial(gz)$ is well defined.  We set
 \begin{equation}\label{definingpartialz}
 \partial(z):=\sum_{j=1}^n\frac{\alpha_j}{t-a_j}z,\quad \text{and}\quad f\p(gz):=f\partial(g)z+fg \partial(z)
 \end{equation}
 where the second equation can be viewed as defining $f\partial (gz)$ with $f$ and $g$ running over basis elements of $R_a$ given in \lemref{lemma1}.   The later defining equation for $f\partial (gz)$ is linear in $f$ and $g$ and gives the definition for all $f,g\in R_a$, not necessarily basis elements.   Consequently we have a well defined linear map $ \W_a\to \text{End}(R_az)$ when $\beta=0$.   In order to prove that $\cdot$ makes $R_az$ into a $\W_a$-module we will need to know that\eqnref{definingbracket} is satisfied for\eqnref{definingpartialz}.  To that end we expand out (using\eqnref{definingpartialz})
\begin{align*}
(g\p(f)-f\p(g))\partial(hz)=(g\p(f)-f\p(g))\left(\partial(h)z+h\p(z)\right)
\end{align*}
while on the other hand (again using\eqnref{definingpartialz})
\begin{align*}
g\p(f\p(hz))-f\p(g\p(hz))&=g\p(f\p(h)z+fh\p(z))\\
 &\quad - f\p(g\p(h)z+gh\p(z)) \\
 &=g\p(f)\p(h)z+gf\p^2(h)z+2fg\p(h)\p(z) \\
&\quad +g\p(f)h\p(z)  +fgh\p^2(z) \\
&\quad -f\p(g)\p(h)z-fg\p^2(h)z-2fg\p(h)\p(z) \\
&\quad -f\p(g)h\p(z)-fgh\p^2(z) \\
&=(g\p(f)-f\p(g))\left(\partial(h)z+h\p(z)\right).
\end{align*}
Thus \eqnref{definingbracket} is satisfied.
We also set $zf(t)=f(t)z$
for any $f\in R_a$.

\begin{lemma} $V(\a,\b)$ is a $\W_a$-module  for all $\a\in\C^n$
and $\b\in\C$.
\end{lemma}

\begin{proof}
Since $R_a$ and $\W_a=R_a\p$ both act on $V(\a,\b)=R_az$, one has to
be very careful about this. Now to see that the above action makes
$V(\a,\b)$ into a $\W_a$-module, we need to check
$$[f\p,g\p]\cdot (hz)=(f\p)\cdot (g\p)\cdot (hz)-(g\p)\cdot(f\p)\cdot(hz),\,\,\forall\, f,g,h\in R_a.$$
Using the fact that \eqnref{definingbracket} is satisfied for our definition\eqnref{definingpartialz}, the right hand side of the above equation is
\begin{align*}
(f\p)\cdot&( g\p(hz)+\b hz\p(g))-(g\p)\cdot (f\p(hz)+\b hz\p(f)) \\
 &=f\p(g\p(hz)+\b hz\p(g))+\b (g\p(hz)+\b hz\p(g))\p(f) \\
&\hskip 2cm -(g\p(f\p(hz)+\b hz\p(f))+\b (f\p(hz)+\b hz\p(f))\p(g)) \\
&=(f\p(g)-g\p(f))\p(hz)+\b hz\p(f\p(g)-g\p(f)) \\
&=(f\p(g)\p-g\p(f)\p)\cdot(hz)={\text{left hand side}}.
\end{align*}
 Thus the $V(\a,\b)$ are $\W_a$-modules for all $\a\in\C^n$
and $\b\in\C$.\end{proof}

It is clear that, when $\a=(0,\cdots ,0)$ and $\b=0$, $V(\a,\b)$ is
the natural representation of $\mathscr{V}_a$ on $R_a$; when
$\a=(0,\cdots ,0)$ and $\b=-1$, $V(\a,\b)$ is the adjoint
representation of $\mathscr{V}_a$. It is also clear that, if
$\a_i-\a'_i\in\Z$ for all $i=1,2,\cdots ,n$ and $\b=\b'$ then
$V(\a,\b)\simeq V(\a',\b')$.

%
%
We call a module of the form $V(\alpha,\beta)$ a {\it module of densities}.

\begin{lemma}\label{submod} Suppose that $S$ is a nonzero subspace of ${V}(\alpha,\b)$
such that $\p(S)\subseteq S$ and $R_aS\subseteq S$. Then
$S=V(\a,\b)$.
\end{lemma}

\begin{proof} Since $S$ is a nonzero $R_a$-submodule of
$R_az$, there exists a nonzero polynomial $g\in R_a$ such that
$gz\in S$. Let $\deg(g)=m$. Noting that
$\p(z)=\sum_{i=1}^n(\a_i/(t-a_i))z$, we  obtain from\eqnref{definingpartialz}
$$\p(g)z=\p(gz)-g\p(z)=\p(gz)-\sum_{i=1}^n(\a_i/(t-a_i))gz\in S.$$ Thus
$\p^m(g)z\in S$ while $\p^m(g)z$ is a nonzero multiple of $z$. Then
$z\in S$.  Consequently, $R_az\subseteq S$ and hence $S=R_az$, as
desired.\end{proof}

Now we can use the above lemma to determine criteria on $(\alpha,\beta)$ for the irreducibility of
$V(\a,\beta)$ and to describe all submodules of $V(\a,\b)$ if it is
reducible.

\begin{theorem}
\begin{enumerate}[(a).]
\item The $\mathscr{V}_a$-module $V(\alpha,\b)$ is reducible if and only if
\begin{enumerate}[(1).]
\item  $\b=0$ and $\a\in\Z^n$;
or
\item $\b=1, n\geq 2$; or
\item $\b=1$, $n=1$ and $\a\in\Z$.
\end{enumerate}
\item If $\b=0$ and $\a\in\Z^n$, then $V(\alpha,\b)$ has only one
nonzero proper submodule which is $1$-dimensional.

\item If $\b=1$, then $V(\a,\b)$ has a smallest nonzero submodule
$\p(R_az)$, which is irreducible; $V(\a,\b)/\p(R_az)$ is a
module with trivial action and is $0$ if and only if $n=1$ and $\a\notin\Z$. If in
addition $\a\in\Z^n$, then
$$\p(R_az)=\sum_{k\in\Z_+}\C t^k\oplus\sum_{i=1}^n\sum_{k\in\N}\C¡¡(t-a_i)^{-k-1},$$
and $\dim(R_az/\p(R_az))=n$.

\item If $\beta=1$ and $\alpha\not\in\mathbb Z^n$, then $\partial ( R_az)\cong V(\alpha,0)$.

\item If $n=1$ and $\alpha\not\in \mathbb Z$, then $V(\alpha, 0)\cong V(\alpha,1)$.
\end{enumerate}
\end{theorem}

\begin{proof}
%
Suppose that $M$ is a nonzero submodule of $V(\a,\b)$. Take any $
g\in R_a$ with $gz\in M$. Formula \eqnref{5.1} shows that
\begin{equation}
f\p(gz)+\b\p(f)gz\in M,\,\,\forall\, f\in R_a, gz\in M,\label{5.2}
\end{equation}
which can also be written as
\begin{equation}
(\b-1)\p(f)gz+\p(fgz)\in M,\,\,\forall\,\, f\in R_a, gz\in M. \label{5.3}
\end{equation}
Taking $f=1$ and $f=t$ in \eqnref{5.2} respectively, we have $\p(gz),
t\p(gz)+\b gz\in M$, i.e.,
\begin{equation}
\p(gz), t\p(gz)\in M\,\,{\text{if}}\,gz\in M.\label{5.4}
\end{equation}
Replacing $gz$ with $t\p(gz)$ in \eqnref{5.3}, we get
\begin{equation}
(\b-1)\p(f)t\p(gz)+\p(ft\p(gz))\in M. \label{5.5}
\end{equation}
Replacing $gz$ with $\p(gz)$ and $f$ with $ft$ in \eqnref{5.3}, we get
\begin{equation}
(\b-1)\p(ft)\p(gz)+\p(ft\p(gz))\in M.\label{5.6}
\end{equation}
Then we subtract \eqnref{5.5} from \eqnref{5.6} to obtain
\begin{equation}
(\b-1)\p(gz)(\p(ft)-\p(f)t)=(\b-1)\p(gz)f\in M,\label{5.7}
\end{equation}
for all $f,g\in R_a$ with $gz\in M$.

\vskip 5pt
\noindent{\bf Case 1.} $\b\neq1$.

First suppose that there is some $hz\in M$ such that $\p(hz)\neq0$.
By \eqnref{5.7}, $R_a\p(hz)$ is a nonzero $R_a$-stable subspace contained
in $M$. By Zorn's lemma there exists a (nonzero) maximal $R_a$-stable subspace $S$
contained in $M$. It is obvious that $\p(S)\subseteq\p(M)\subseteq
M$ by \eqnref{5.4} and hence $\p(S)+S\subseteq M$. Take any $f,g\in R_a$
with $gz\in S$. Then
$$f\p(gz)=\p(fgz)-\p(f)gz\in\p(S)+S,$$
which gives $R_a\p(S)\subseteq\p(S)+S$ and hence
$R_a(\p(S)+S)\subseteq\p(S)+S$. Thus by the maximality of $S$, we
have $(\p(S)+S)\subseteq S$ and in particular $\p(S)\subseteq S$.
Thus $S=V(\a,\b)$ by \lemref{submod}, which implies $M=V(\a,\b)$.

\vskip 5pt \noindent{\bf Subcase 1.1.} $\a\in\Z^n$ and $\b=0$.

In this case, $z\in R_a^*$, and it is clear that $\C$ is a trivial
submodule of $V(\a,\b)$. If $M\neq \C$, then there exists some
$hz\in M$ such that $\p(hz)\neq0$, and hence $M=V(\a,\b)$ by the
previous discussion. This proves (b).

\vskip 5pt \noindent{\bf Subcase 1.2.} $\a\notin\Z^n$ or $\b\neq0$.

If $\p(hz)=0$ for all $hz\in M$, then $M=\C$, $z\in R_a^*$ and hence
$\a\in\Z^n$. By taking $f=t^2, g=z^{-1}$ in \eqnref{5.2}, we get $2\b t\in M$,
forcing $\b=0$, contradiction. Thus we must have $\p(hz)\neq0$ for
some $hz\in M$. Then
 $M=V(\a,\b)$.

\vskip 5pt
\noindent{\bf Case 2.} $\b=1$.

Consider $R_aM\neq0$, which is an $R_a$-module. Given any $f,g\in
R_a$ with $gz\in M$, from \eqnref{5.3} we know that $\p(fgz)\in M\subset
R_aM$, that is, $\p(R_aM)\subseteq R_aM$. Then by \lemref{submod}, we get
$R_aM=V(\a,\b)$, or in other words, $R_aM=R_az$.

Note that\eqnref{5.3} becomes $\p(fgz)\in M$ for any $f\in R_a$ and $gz\in
M$ in this case. This indicates that $\p(R_az)=\p(R_aM)\subseteq M$.
On the other hand, from\eqnref{5.1} we have
\begin{equation}
(f\p)\cdot (gz)=\p(fgz)\in\p(R_az),\,\,\,\forall \,\,f,g\in R_a,\label{5.8}
\end{equation}
which shows that $\p(R_az)$ is stable under the action of $\W_a$.
Moreover $\p(R_az)$ is an irreducible $\W_a$-module as it is contained in any nonzero $\W_a$-submodule $M$. From\eqnref{5.8}, we
see that the quotient module $V(\a,\b)/\p(R_az)$ is always trivial.

When $\a\in\Z^n$, we have obviously that
$$\p(R_az)=\sum_{k\in\Z_+}\C t^k\oplus\sum_{i=1}^n\sum_{k\in\N}\C¡¡(t-a_i)^{-k-1},$$
with $\dim(R_az/\p(R_az))=n$.

To complete the proof of (a) and (c), we need only to show
\medskip

\noindent{\bf Claim.} $\p(R_az)=R_az$ if and only if $n=1$ and
$\a\notin\Z$.
\medskip

\noindent{\it Proof of Claim.} When $n=1$, $\W_a$ is the classical
Virasoro algebra and $V(\a,\b)$ is just the intermediate series
modules of $\W_a$ (see [KR]). Thus the claim for $n=1$ is obvious
from the representation theory of Virasoro algebra.

Now we suppose that $n\geq 2$, and without loss of generality we may
assume that $\hbox{Re}(\a_i)<1$ for all $i=1,...,n$, where
$\hbox{Re}(\a_i)$ is the real part of $\a_i$. Suppose contrary to the
claim, we have $\p(R_az)=R_az$. Then there exists
$f=\sum_{j=0}^mb_jt^j\in\C[t]$ and
$f_i=\sum_{j=1}^{m_i}b_{i,j}(t-a_i)^{-j},\,i=1,...,n$ such that
$\p((f+\sum_{i=1}^nf_i)z)=z$, that is,
\begin{align}
\left(\p(f)+\sum_{i=1}^n\sum_{j=1}^{m_i}\frac{-jb_{i,j}}{(t-a_i)^{j+1}}\right)&z\ +\left(f+\sum_{i=1}^n\sum_{j=1}^{m_i}\frac{b_{ij}}{(t-a_i)^{j}}\right)&\sum_{i=1}^n\frac{\a_i}{t-a_i}z=z.\label{5.9}
\end{align}
If $f_i\neq0$ for some $1\leq i\leq n$, we may suppose
$b_{i,m_i}\neq0$, then comparing the coefficients of
$(t-a_a)^{-m_i-1}z$ in \eqnref{5.9}, we have
$-m_ib_{i,m_i}+\a_ib_{i,m_i}=0$, which implies $\a_i=m_i\in\N$,
contradicting the assumption $\hbox{Re}(\a_i)<1$. Thus  $f_i=0$ for
all $1\leq i\leq n$, i.e., \eqnref{5.9} becomes
\begin{equation}
\p(f)z+f\sum_{i=1}^n\frac{\a_i}{t-a_i}z=z\,\,\,{\rm \,\,\,\hspace{.5cm}or\hspace{.5cm}\,\,\,\,\,\,}\p(f)+f\sum_{i=1}^n\frac{\a_i}{t-a_i}=1,\label{5.10}
\end{equation}
which implies that $t-a_i$ divides $f$ for any $1\leq i\leq n$.
Since $\p(fz)=z$ infers $f\neq0$, then there exists some nonzero
$h(z)\in\C[t]$ such that $f=(t-a_1)(t-a_2)...(t-a_n)h$. In
particular, $\deg(f)\geq n$. Thus we may assume that $b_m\neq 0$ and
$\deg(f)=m\geq n\geq2$. Equation \eqnref{5.10} can be reformulated
explicitly as
$$
\left(\sum_{j=1}^mjb_jt^{j-1}\right)+\left(\sum_{j=0}^mb_jt^j\right)\left(\sum_{i=1}^n\frac{\a_i}{t-a_i}\right)=1
.$$
Multiply the above equation by $(t-a_1)(t-a_2)...(t-a_n)$, we get
\begin{equation*}
\prod_{i=1}^n(t-a_i)\left(\sum_{j=1}^mjb_jt^{j-1}\right)
 +\left(\sum_{j=0}^mb_jt^j\right)\left(\sum_{i=1}^n\a_i \prod_{k=1,k\neq i}^n(t-a_k)\right)  =\prod_{i=1}^n(t-a_i).
\end{equation*}
Again comparing the coefficients of $t^{n+m-1}$ in \eqnref{5.9} and
noticing $m>1$, we have $-mb_{m}+\sum_{i=1}^n\a_ib_{m}=0$, which
implies $\sum_{i=1}^n\a_i=m\geq n$, contradicting the assumption
$\hbox{Re}(\a_i)<1$. Thus $\p(R_az)\ne R_az$. The claim is proved
and the theorem follows.
%
%
%
%

(d).  In this case one can easily verify that the following linear map is a module isomorphism:
$$\varphi:  V(\alpha,0)\to \partial ( R_az),\,\,\, fz\mapsto \partial ( fz),$$
for any $f\in R_a$.

(e).  Follows from (c) and (d).
\end{proof}

\begin{remark} We do not know the dimension of the module $R_az/\p(R_az)$ with trivial action when $\beta=1$, $n>1$ and $\alpha\not\in \mathbb Z^n$.
\end{remark}

\begin{remark} As we said at the beginning of the paper, the centerless  $n$-point Virasoro algebras, ${\W_a}$,
 are genus zero Krichever-Novikov type algebras. For positive genus  Krichever-Novikov type
 algebras, one can expect to study similar properties as this paper
 does. But we have noticed that the positive genus  Krichever-Novikov type
 algebras behave quite differently from genus zero Krichever-Novikov type algebras, and it is even harder to obtain the same results.
 Now we are working on genus $1$ and $2$  Krichever-Novikov type
 algebras in another paper.
\end{remark}

\begin{remark} Representations in Section 5 are for centerless  $n$-point Virasoro algebras, ${\W_a}$.
It will be very interesting to investigate representations for
$n$-point Virasoro algebras with nonzero central actions. One may
study highest-weight-like modules similar to usual highest weight
modules. But note that the algebra $\tilde{\W_a}$ does not have a
good $\Z$-gradation or a triangular decomposition as defined in
\cite{MP}.
\end{remark}
\begin{remark}. The restriction of the module structure on $V(\alpha, \beta)$ to the centerless Virasoro algebra $\Der(C[t,t^{-1}]$ is completely determined in \cite{MR3056686}.
\end{remark}

\section*{Acknowledgments}

B.C. was partially supported by a College of Charleston faculty research and development grant; X.G. is partially supported by NSF of China (Grant 11101380); R.LuÊ is partially supported by NSF of China (Grant 11371134) and Jiangsu Government Scholarship for Overseas Studies (JS-2013-313); K.Z. is partially supported by NSF of China (Grant 11271109) and NSERC. The first author would like to thank K.Z. and Wilfred Laurier University for their hospitality during his summer 2010 visit to Canada where this project began. All authors like to thank the referee for corrections and good suggestions. 

\def\cprime{$'$} \def\cprime{$'$} \def\cprime{$'$}
\providecommand{\bysame}{\leavevmode\hbox to3em{\hrulefill}\thinspace}
\providecommand{\MR}{\relax\ifhmode\unskip\space\fi MR }

\providecommand{\MRhref}[2]{%
  \href{http://www.ams.org/mathscinet-getitem?mr=#1}{#2}
}
\providecommand{\href}[2]{#2}

\appendix
\section{}\label{examples}
We give in this section examples of various $a$  that give the
automorphism groups of Klein listed in Theorem 13 as a series of
lemmas.
\begin{lemma}\label{ringisolem}
Let $n\ge4$ and $a=(a_1, a_2, \cdots, a_n)$ where $ a_{k}=\zeta^{k}$
for the primitive $n$-th root of unity $\zeta=\exp(2\pi \imath /n)$.  Then
Aut(${\mathcal{V}}_a)\cong C_n$, the cyclic group of order $n$.
\end{lemma}

\begin{lemma}
Let $n\ge3$, and $a=(0, a_1, a_2, \cdots, a_n)$ where $
a_{k}=\zeta^{k}$ for a primitive $n$-th root of unity $\zeta$.
\begin{enumerate}
\item[(a).]
 If $n\ne 4$ then  Aut(${\mathcal{V}}_a)\cong D_n$, the dihedral
group of order $2n$. \item[(b).] If $n= 4$ then
Aut(${\mathcal{V}}_a)\cong S_4$.\end{enumerate}
\end{lemma}

\begin{lemma}[\cite{MR1863996}, \cite{MR0080930}]  Let $\zeta=\exp (2\pi \imath  /5)$ and
$$
a_0=0, \enspace a_i= \zeta^{i-1}(\zeta+\zeta^4),\enspace 1\leq i\leq
5, \quad a_i=\zeta^{i-6}(\zeta^2+\zeta^3),\enspace 6\leq i\leq 10.
$$
The automorphism group of $\W_{(0,a_1,\dots, a_{10})}$ is
$A_5$ where  the automorphisms are
\begin{gather*}
t\mapsto \zeta^jt,\quad t\mapsto -\frac{1}{\zeta^j t}, \\
t\mapsto \zeta^j\frac{-(\zeta-\zeta^4)\zeta^lt +(\zeta^2-\zeta^3)}{(\zeta^2-\zeta^3)\zeta^lt+(\zeta-\zeta^4)}  \\
t\mapsto \zeta^j\frac{(\zeta^2-\zeta^3)\zeta^lt+(\zeta-\zeta^4)}
{(\zeta-\zeta^4)\zeta^lt -(\zeta^2-\zeta^3)},
\end{gather*}
for $j,l=0,\dots 4$.
\end{lemma}
\begin{proof}  The proof follows from results in \cite{MR0080930} and
\cite{MR1427489}.
\end{proof}


\begin{thebibliography}{0}


\bibitem{MR1261553}
Murray Bremner, \emph{Generalized affine {K}ac-{M}oody {L}ie algebras over
  localizations of the polynomial ring in one variable}, Canad. Math. Bull.
  \textbf{37} (1994), no.~1, 21--28. \MR{MR1261553 (95d:17025)}

\bibitem{MR918402}
A.~Cappelli, C.~Itzykson, and J.-B. Zuber.
\newblock The {${\rm A}$}-{${\rm D}$}-{${\rm E}$} classification of minimal and
  {$A^{(1)}_1$} conformal invariant theories.
\newblock {\em Comm. Math. Phys.}, 113(1):1--26, 1987.

\bibitem{CF1}
Ben Cox and Vyacheslav Futorny, \emph{DJKM algebras I: Their universal central
  extension}, Proc. Amer. Math. Soc. \textbf{139} (2011), 3451--3460.

\bibitem{MR1104280}
B.~L. Feigin and D.~B. Fuchs, \emph{Representations of the {V}irasoro
  algebra}, Representation of {L}ie groups and related topics, Adv. Stud.
  Contemp. Math., vol.~7, Gordon and Breach, New York, 1990, pp.~465--554.
  \MR{1104280 (92f:17034)}

\bibitem{MR1849359}
Edward Frenkel and David Ben-Zvi, \emph{Vertex algebras and algebraic curves},
  Mathematical Surveys and Monographs, vol.~88, American Mathematical Society,
  Providence, RI, 2001. \MR{1 849 359}
  
 \bibitem{MR3056686} X. Guo, R. Lu and K. Zhao, \emph{Fraction representations and highest-weight-like representations of the Virasoro algebra}, J. Algebra, 387(2013), 68-86. \MR{MR3056686}

\bibitem{MR829385}
D.~A. Jordan, \emph{On the ideals of a {L}ie algebra of derivations}, J. London
  Math. Soc. (2) \textbf{33} (1986), no.~1, 33--39. \MR{829385 (87c:17020)}

\bibitem{MR1021978}
V.~G. Kac and A.~K. Raina, \emph{Bombay lectures on highest weight
  representations of infinite-dimensional {L}ie algebras}, Advanced Series in
  Mathematical Physics, vol.~2, World Scientific Publishing Co. Inc., Teaneck,
  NJ, 1987. \MR{1021978 (90k:17013)}

\bibitem{MR678001}
Irving Kaplansky, \emph{The {V}irasoro algebra}, Comm. Math. Phys. \textbf{86}
  (1982), no.~1, 49--54. \MR{678001 (84c:17012)}

\bibitem{MR1104840}
David Kazhdan and George Lusztig, \emph{Affine {L}ie algebras and quantum
  groups}, Internat. Math. Res. Notices (1991), no.~2, 21--29. \MR{MR1104840
  (92g:17015)}

\bibitem{MR1280096}
E.~Kirkman, C.~Procesi, and L.~Small, \emph{A {$q$}-analog for the {V}irasoro
  algebra}, Comm. Algebra \textbf{22} (1994), no.~10, 3755--3774. \MR{1280096
  (96b:17016)}

\bibitem{MR0080930}
Felix Klein, \emph{Lectures on the icosahedron and the solution of equations of
  the fifth degree}, revised ed., Dover Publications Inc., New York, N.Y.,
  1956, Translated into English by George Gavin Morrice. \MR{0080930 (18,329c)}


\bibitem{MR604577}
John McKay.
\newblock Graphs, singularities, and finite groups.
\newblock In {\em The {S}anta {C}ruz {C}onference on {F}inite {G}roups ({U}niv.
  {C}alifornia, {S}anta {C}ruz, {C}alif., 1979)}, volume~37 of {\em Proc.
  Sympos. Pure Math.}, pages 183--186. Amer. Math. Soc., Providence, R.I.,
  1980.

\bibitem{MR925072}
Igor~Moiseevich Krichever and S.~P. Novikov, \emph{Algebras of {V}irasoro type,
  {R}iemann surfaces and strings in {M}inkowski space}, Funktsional. Anal. i
  Prilozhen. \textbf{21} (1987), no.~4, 47--61, 96. \MR{925072 (89f:17020)}

\bibitem{MR902293}
\bysame, \emph{Algebras of {V}irasoro type, {R}iemann surfaces and the
  structures of soliton theory}, Funktsional. Anal. i Prilozhen. \textbf{21}
  (1987), no.~2, 46--63. \MR{MR902293 (88i:17016)}

\bibitem{MR998426}
\bysame, \emph{Algebras of {V}irasoro type, the energy-momentum tensor, and
  operator expansions on {R}iemann surfaces}, Funktsional. Anal. i Prilozhen.
  \textbf{23} (1989), no.~1, 24--40. \MR{998426 (90k:17049)}

\bibitem{MP}
R.~V.~Moody, A.~Pianzola,  \emph{Lie algebras with triangular
decompositions}, Canad. Math. Soc., Ser. Mono. Adv. Texts, A
Wiley-Interscience Publication, John Wiley \& Sons Inc., New York,
1995.

\bibitem{MR1666274}
M.~Schlichenmaier and O.~K. Scheinman, \emph{The {S}ugawara construction and
  {C}asimir operators for {K}richever-{N}ovikov algebras}, J. Math. Sci. (New
  York) \textbf{92} (1998), no.~2, 3807--3834, Complex analysis and
  representation theory, 1. \MR{1666274 (2000g:17036)}

\bibitem{MR1058993}
Martin Schlichenmaier, \emph{Central extensions and semi-infinite wedge
  representations of {K}richever-{N}ovikov algebras for more than two points},
  Lett. Math. Phys. \textbf{20} (1990), no.~1, 33--46. \MR{1058993 (92c:17035)}

\bibitem{MR1039524}
\bysame, \emph{Krichever-{N}ovikov algebras for more than two points}, Lett.
  Math. Phys. \textbf{19} (1990), no.~2, 151--165. \MR{1039524 (91a:17039)}

\bibitem{Schl1}
\bysame, \emph{Verallgemeinerte Krichever-Novikov Algebren und
deren Darstellungen}, Ph.D. thesis, Universit\"at Mannheim, 1990.

\bibitem{MR2058804}
\bysame, \emph{Higher genus affine algebras of {K}richever-{N}ovikov type},
  Mosc. Math. J. \textbf{3} (2003), no.~4, 1395--1427. \MR{2058804
  (2005f:17025)}

\bibitem{MR1989644}
\bysame, \emph{Local cocycles and central extensions for multipoint algebras of
  {K}richever-{N}ovikov type}, J. Reine Angew. Math. \textbf{559} (2003),
  53--94. \MR{1989644 (2004c:17056)}

\bibitem{MR1706819}
M.~Shlichenmaier and O.~K. Sheinman, \emph{The
  {W}ess-{Z}umino-{W}itten-{N}ovikov theory, {K}nizhnik-{Z}amolodchikov
  equations, and {K}richever-{N}ovikov algebras}, Uspekhi Mat. Nauk \textbf{54}
  (1999), no.~1(325), 213--250. \MR{MR1706819 (2001f:81060)}

\bibitem{MR2106647}
\bysame, \emph{The {K}nizhnik-{Z}amolodchikov equations for positive genus, and
  {K}richever-{N}ovikov algebras}, Uspekhi Mat. Nauk \textbf{59} (2004),
  no.~4(358), 147--180. \MR{2106647 (2005k:32016)}
\bibitem{MR2072650}
O.~K. Sheinman, \emph{Second-order {C}asimirs for the affine
  {K}richever-{N}ovikov algebras {$\widehat{\mathfrak g\mathfrak l}_{g,2}$} and
  {$\widehat{\mathfrak s\mathfrak l}_{g,2}$}}, Fundamental mathematics today
  ({R}ussian), Nezavis. Mosk. Univ., Moscow, 2003, pp.~372--404. \MR{2072650
  (2005i:17029)}

\bibitem{MR2152962}
\bysame, \emph{Highest-weight representations of {K}richever-{N}ovikov algebras
  and integrable systems}, Uspekhi Mat. Nauk \textbf{60} (2005), no.~2(362),
  177--178. \MR{2152962 (2006b:17041)}


\bibitem{MR1427489}
Jerry Shurman, \emph{Geometry of the quintic}, A Wiley-Interscience
  Publication, John Wiley \& Sons Inc., New York, 1997. \MR{1427489
  (97i:12002)}


\bibitem{MR966871}
S.~M. Skryabin, \emph{Regular {L}ie rings of derivations}, Vestnik Moskov.
  Univ. Ser. I Mat. Mekh. (1988), no.~3, 59--62. \MR{966871 (90c:13026)}

\bibitem{MR2035385}
Serge Skryabin, \emph{Degree one cohomology for the {L}ie algebras of
  derivations}, Lobachevskii J. Math. \textbf{14} (2004), 69--107 (electronic).
  \MR{2035385 (2004m:17030)}

\bibitem{MR1243417}
Karlheinz Spindler, \emph{Abstract algebra with applications. {V}ol. {II}},
  Marcel Dekker Inc., New York, 1994, Rings and fields. \MR{1243417
  (94i:00002b)}

\bibitem{MR1863996}
Gabor Toth, \emph{Finite {M}\"obius groups, minimal immersions of spheres, and
  moduli}, Universitext, Springer-Verlag, New York, 2002. \MR{1863996
  (2002i:53082)}

\bibitem{MR1682304}
Friedrich Wagemann, \emph{Some remarks on the cohomology of
  {K}richever-{N}ovikov algebras}, Lett. Math. Phys. \textbf{47} (1999), no.~2,
  173--177. \MR{1682304 (2000e:17022)}
\end{thebibliography}
\end{document}